\documentclass[12pt]{article}

\usepackage{amsmath,amsthm,amssymb}
\allowdisplaybreaks
\usepackage{mathrsfs}
\usepackage{cite}

\usepackage[pdftex,bookmarksopen=true,bookmarks=true,unicode,setpagesize]{hyperref}
\hypersetup{colorlinks=true,linkcolor=black,citecolor=black}

\newtheorem{theorem}{Theorem}
\newtheorem{corollary}[theorem]{Corollary}
\newtheorem{lemma}[theorem]{Lemma}
\newtheorem{proposition}[theorem]{Proposition}

\theoremstyle{remark}

\theoremstyle{remark}

\theoremstyle{remark}
\newtheorem{remark}[theorem]{Remark}

\newcommand{\R}{\mathbb R}

\begin{document}

\vspace{-20mm}
\begin{center}{\Large \bf
Approximation of a free Poisson process by systems of freely independent particles
}
\end{center}

{\large Marek Bo\.zejko}\\
Instytut Matematyczny, Uniwersytet Wroc{\l}awski, Pl.\ Grunwaldzki 2/4, 50-384 Wroc{\l}aw, Poland; e-mail: \texttt{bozejko@math.uni.wroc.pl}\vspace{2mm}

{\large Jos\'e  Lu\'is da Silva}\\
Universidade da Madeira, Edif\'icio da Penteada, Caminho da Penteada, 9020-105, Funchal, Madeira, Portugal; e-mail: \texttt{luis@uma.pt}\vspace{2mm}

{\large Tobias Kuna}\\ University of Reading,
Department of  Mathematics,
Whiteknights,
PO Box 220,
Reading RG6 6AX, U.K.\\
 e-mail:
\texttt{t.kuna@reading.ac.uk}\vspace{2mm}

{\large Eugene Lytvynov}\\ Department of Mathematics,
Swansea University, Singleton Park, Swansea SA2 8PP, U.K.;
e-mail: \texttt{e.lytvynov@swansea.ac.uk}\vspace{2mm}


{\small

\begin{center}
{\bf Abstract}
\end{center}
\noindent 
Let $\sigma$ be a non-atomic, infinite Radon measure on $\R^d$, for example, $d\sigma(x)=z\,dx$ where $z>0$. 
We consider a system of  freely independent particles $x_1,\dots,x_N$ in a bounded set $\Lambda\subset\R^d$, where 
each particle $x_i$ has distribution $\frac1{\sigma(\Lambda)}\,\sigma$ on $\Lambda$ and the number of particles, $N$, is random and has Poisson distribution with parameter $\sigma(\Lambda)$. If the particles were classically independent rather than freely independent, this particle system would be the restriction to $\Lambda$ of the Poisson point process on $\R^d$ with intensity measure $\sigma$. In the case of free independence, this particle system is not the restriction of the free Poisson process on $\R^d$ with intensity measure $\sigma$. Nevertheless, we prove that this is true in an approximative sense: if bounded sets $\Lambda^{(n)}$ ($n\in\mathbb N$) are such that $\Lambda^{(1)}\subset\Lambda^{(2)}\subset\Lambda^{(3)}\subset\dotsm$ and $\bigcup_{n=1}^\infty \Lambda^{(n)}=\R^d$,  then the corresponding particle system in $\Lambda^{(n)}$ converges (as $n\to\infty$) to the free Poisson process on $\R^d$ with intensity measure $\sigma$. We also prove the following $N/V$-limit: Let $N^{(n)}$ be a determinstic sequence of natural numbers such that $\lim_{n\to\infty}N^{(n)}/\sigma(\Lambda^{(n)})=1$. 
Then the system of $N^{(n)}$ freely independent particles in $\Lambda^{(n)}$ converges 
(as $n\to\infty$) to the free Poisson process. We finally extend these results to the case of a free L\'evy white noise  (in particular, a free L\'evy process) without free Gaussian part.
 } \vspace{2mm}

 {\bf Keywords:} Completely random measure; freely independent particle systems; free Poisson process; free L\'evy process; $N/V$-limits.\vspace{2mm}

{\bf 2010 MSC:} 46L54, 60G20, 60G51, 60G55, 60G57, 82B21.

\section{Introduction}
Let $\Gamma(\R^d)$ denote the configuration space of an (infinite) system of identical particles in $\R^d$. By definition, $\Gamma(\R^d)$ is the collection of all sets $\gamma\subset\R^d$ which are locally finite. A probability measure on $\Gamma(\R^d)$ is called a point process, see e.g.\ \cite{Kallenberg}. In statistical mechanics, a point process describes the state of a system of particles in the continuum. Although every realistic system has only a finite number of particles, many important physical phenomena can only be described by infinite particle systems, see e.g.\ \cite{Ruelle,DSS}.

An infinite system of particles without interaction is described by a Poisson point process.
Let $\sigma$ be a non-atomic, infinite Radon measure on $\R^d$, for example $d\sigma(x)=z\,dx$, where $z>0$.
A Poisson point process on $\R^d$ with intensity measure $\sigma$ is the probability measure $\pi$ on $\Gamma(\R^d)$ whose Fourier transform is given by the formula \eqref{vtdry7r} below.  The measure $\pi$ can also be uniquely characterized by the following property: For $\Lambda\in\mathscr B_0(\R^d)$,  the number of particles of a configuration $\gamma$ which belong to $\Lambda$  has Poisson distribution   with parameter $\sigma(\Lambda)$, and given that $\gamma\cap\Lambda=\{x_1,\dots,x_N\}$ the particles $x_1,\dots,x_N$ are independent and the distribution of each particle $x_i$ on  $\Lambda$ is $\frac1V\,\sigma$.
Here, $\mathscr B_0(\R^d)$ denotes the collection of all bounded Borel subsets of $\R^d$, and  $V=\sigma(\Lambda)$ is  the volume of $\Lambda$.

The Poisson point process $\pi$ can  be derived through the $N/V$-limit. More precisely, let us consider an isotone  sequence of sets $\Lambda^{(n)}\in\mathscr B_0(\R^d)$, i.e., $\Lambda^{(1)}\subset\Lambda^{(2)}\subset\Lambda^{(3)}\subset\dotsm$, and assume that $\bigcup_{n=1}^\infty\Lambda^{(n)}=\R^d$. Denote $V^{(n)}:=\sigma(\Lambda^{(n)})$ and let $N^{(n)}$ be natural numbers such that the condition \eqref{nig8rf} below is satisfied. Then the Poisson point process $\pi$ is equal to the limit (as $n\to\infty$)  of  the system of $N^{(n)}$ independent particles in $\Lambda^{(n)}$ such that the distribution of each particle on $\Lambda^{(n)}$ is $\frac1{V^{(n)}}\,\sigma$.

This result admits a  generalization to  measure-valued L\'evy processes on $\R^d$, which are  probability measures on $\mathbb K(\R^d)$, e.g.\ \cite{KingmanCRM,Kingman,KLV}. Here $\mathbb K(\R^d)$ denotes the space of all discrete signed Radon measures on $\R^d$, i.e., signed Radon measures of the form $\sum_i s_i\delta_{x_i}$, where $s_i\ne0$ and  $\delta_{x_i}$ denotes the Dirac measure with mass at $x_i$. The points $x_i$ may be interpreted as locations of  particles (or some organisms in biological interpretation). The weights $s_i$ are a certain attribute attached to these particles (organisms). Note that the set $\{x_i\}$ is not necessarily locally finite. For a large class of measure-valued L\'evy processes, $\{x_i\}$  is even dense in $\R^d$ a.s.\ \cite{KLV}.

In free probability, the notion of  independence of random variables is replaced by Voiculescu's definition of free independence of noncommutative prob\-ability spaces. This has led to a deep theory, in which many results are noncommutative analogs of fundamental results of classical probability theory, see e.g.\
\cite{NiSp,V,Voiculescu}. Let us  specifically mention the paper \cite{BP}, which studies stable laws and domains of attraction in free probability theory and establishes  the  Bercovici--Pata bijection between the  infinitely divisible distributions on $\R$ and the freely infinitely divisible distributions on $\R$. See also \cite{BNT,Biane,V} for  discussions of free L\'evy processes.

An analog of the classical Poisson distribution on $\R$ with parameter $\alpha$ is the free Poisson distribution with parameter
$\alpha$, also known as  the Marchenko--Pastur distribution. This distribution can be derived as the limit (as $n\to\infty$) of the $n$th  free convolution power of the distribution $(1-\frac{\alpha}n)\delta_0+\frac{\alpha}n\delta_1$, where $\alpha>0$ is a parameter of the distribution. It should be stressed that the free Poisson distribution is not discrete: it has density with respect to the Lebesgue measure and possibly one atom, see e.g.\ \cite{NiSp}.

As shown in \cite{BNT,V}, the Poisson point process with intensity $\sigma$ also has a counterpart in free probability.
According to \cite{BL}, the free Poisson process on $\R^d$ is a real algebra generated by a family of bounded selfadjoint operators $A(f)$ in the full Fock space over $L^2(\R^d,\sigma)$, and the vacuum state on this algebra plays the role of an expectation.  Here $f$ belongs to the class of bounded measurable functions on $\R^d$ with compact support. The operators $A(f)$ resemble the commuting operators in the symmetric Fock space over $L^2(\R^d,\sigma)$ which identify the Poisson point process, see e.g.\  \cite{L,Surgailis}.

In the case of the one-dimensional underlying space $\R$, Ben Arous and Kargin \cite{BAK} gave a rigorous meaning to the heuristic notion of a system of $N$ freely independent, identically distributed  particles.  Furthermore, it  follows from \cite[Theorem~2]{BAK}  that the $N/V$-limit holds for a free Poisson process on $\R$. More precisely, if the condition \eqref{nig8rf} below is satisfied, the free Poisson process with intensity measure $\sigma$  can be approximated by a system of $N^{(n)}$ freely independent particles in $\Lambda^{(n)}$ such that the distribution of each particle on $\Lambda^{(n)}$ is $\frac1{V^{(n)}}\,\sigma$.

The first result of the present paper states that the $N/V$-limit holds for a free Poisson process on any underlying space $\R^d$ (and even on any locally compact Polish space). We prove the convergence of the corresponding  moments, which is stronger than the weak convergence  shown in \cite{BAK} in the case of the underlying space $\R$. Furthermore, our proof of this result is significantly shorter and simpler than the proof in \cite{BAK}. We also extend the  $N/V$-limit to the case of a free L\'evy white noise as defined in \cite{BL}. In particular, such a result holds for a free L\'evy process without free Gaussian part \cite{BNT}.

The  main result of the paper concerns  a noncommutative probability space
$(\mathscr A(\Lambda),\tau_\Lambda)$ which describes a system of $N$ freely independent particles in a set $\Lambda\in\mathscr B_0(\R^d)$ which have distribution $\frac{\sigma}{V}$, where the number of particles, $N$, is random and has the (classical) Poisson distribution with parameter $\sigma(\Lambda)$. Here $\mathscr A(\Lambda)$ is an algebra of (unbounded) linear operators in a certain Hilbert space and $\tau_\Lambda$ is a trace on this algebra. As mentioned above, if the particles were independent rather than freely independent, then this particle system would be precisely the restriction to $\Lambda$ of the Poisson point process on $\R^d$ with intensity measure $\sigma$. However, in our noncommutative setting, $(\mathscr A(\Lambda),\tau_\Lambda)$ is not the restriction to  $\Lambda$ of the free Poisson process with intensity $\sigma$. Nevertheless, we prove that, if $(\Lambda^{(n)})_{n=1}^\infty$ is an isotone sequence of  sets from $\mathscr B_0(\R^d)$ whose union is $\mathbb R^d$, then the probability spaces $(\mathscr A(\Lambda^{(n)}),\tau_{\Lambda^{(n)}})$ converge, in the sense of moments, to the free Poisson process on $\R^d$ with intensity $\sigma$.

The latter result can be interpreted as follows. Let $\Theta,\Lambda\in\mathscr B_0(\R^d)$ be such that $\Theta\subset\Lambda$ and the set $\Lambda$ is much larger than $\Theta$. Then the restriction to $\Theta$ of the noncommutative probability space
$(\mathscr A(\Lambda),\tau_\Lambda)$ is very close to the restriction to $\Theta$ of the free Poisson process on $\R^d$ with intensity $\sigma$.
We also extend this result to the case of a free L\'evy white noise (in particular, a free L\'evy process).

We stress that, for the results of this paper, it is principal that both the underlying space and the intensity measure  of the free Poisson process are infinite.

The paper is organized as follows. In Section \ref{hklhyi}, we briefly recall some definitions and properties
related to the classical Poisson point process,  a measure-valued L\'evy process, and a L\'evy white noise. Although we do  not use these results directly in this paper,  they are needed to better understand analogy and differences between the classical and the free cases. In Section~\ref{niugt8i} we recall the definition of free independence, namely in terms of cumulants which is more convenient for applications. Moreover we define the free L\'evy white noise. In Section~\ref{jkgt8yt}, we construct noncommutative probability spaces which describe systems of freely independent, identically distributed particles, and free probability counterparts of discrete random measures $\sum_{i=1}^N s_i\delta_{x_i}$ such that the pairs $(s_i,x_i)$ are independent and identically distributed. In Section \ref{lngity8t}, we formulate the results of the paper. Finally, in Section \ref{yufu8} we prove the results.

\section{A measure-valued L\'evy process and a L\'evy white noise}\label{hklhyi}

Let $X$ be a locally compact Polish space.  Let $\mathscr B(X)$ denote the Borel $\sigma$-algebra on $X$, and  let $\mathscr B_0(X)$ denote the collection of all Borel sets in $X$ which have compact closure. A measure on $(X,\mathscr B(X))$ is called {\it Radon} if it takes finite values on each set from $\mathscr B_0(X)$. A {\it signed Radon measure on $X$} is a difference of two Radon measures on $X$. Let $\mathbb M(X)$ denote the collection of all signed Radon measures on $X$.
The space $\mathbb M(X)$ is equipped with the vague topology, and let $\mathscr B(\mathbb M(X))$ denote the corresponding Borel $\sigma$-algebra.

The {\it configuration space} $\Gamma(X)$ is defined as the collection  of all sets $\gamma\subset X$ that are locally finite. Usually, a configuration $\gamma\in\Gamma(X)$ is identified with the measure $\gamma=\sum_{x\in\gamma}\delta_x$, where $\delta_x$ denotes the Dirac measure with mass at $x$. Under this identification, the configuration $\gamma$ becomes a Radon measure on $X$, so we have the inclusion $\Gamma(X)\subset\mathbb M(X)$. Below it should always be clear from the context which of the two interpretations of a configuration we are currently using.

A measurable mapping $\eta$ from a probability space into $\mathbb M(X)$ is called a {\it (signed) random measure on $X$}, see e.g.\ \cite{Kallenberg, Kingman}. If $\eta$ takes a.s.\ values in $\Gamma(X)$, then $\eta$ is called a  {\it (simple) point process}.
(In particular, any probability measure on $\mathbb M(X)$ or $\Gamma(X)$ identifies a random measure or a point process, respectively.)

A random  measure $\eta$ is called {\it completely random} if, for any mutually disjoint sets $\Lambda_1,\dots,\Lambda_n\in\mathscr B_0(X)$,
the random variables $\eta(\Lambda_1),\dots,\eta(\Lambda_n)$  are independent \cite{KingmanCRM,Kingman}.

Let us fix a reference measure $\sigma$ on $X$. We assume that $\sigma$ is a non-atomic Radon measure and  $\sigma(X)=\infty$. 

Assume that a completely random measure $\eta$ has in addition  the property that, for any two sets $\Lambda_1,\Lambda_2\in\mathscr B_0(X)$ such that
$\sigma(\Lambda_1)=\sigma(\Lambda_2)$, the random variables $\eta(\Lambda_1)$ and $\eta(\Lambda_2)$ have the same distribution. Then, in e.g.\ \cite{KLV,TsVY}, such a random measure was called a {\it  measure-valued L\'evy process on $X$}. If $X=\R$, $\sigma$ is the Lebesgue measure,  and $L_t:=\eta([0,t])$, then $(L_t)_{t\ge0}$ is just a classical L\'evy process of bounded variation. The following proposition describes the measure-valued L\'evy process on $X$ in terms of their Fourier transform.

\begin{proposition}[\cite{KingmanCRM,Kingman}]\label{vuytd7i}
(i) Let $\eta$ be a random measure on $X$, and let  $\mu$ denote the distribution of $\eta$. Then $\eta$ is a measure-valued L\'evy process on $X$ if and only if there exists a measure $\nu$ on $\R^*:=\R\setminus\{0\}$ which satisfies 
\begin{equation}\label{utf7r}
\int_{\R^*}\min\{1,|s|\}\,d\nu(s)<\infty,\end{equation} 
and the Fourier transform of $\mu$ is given by 
\begin{equation}\label{tyr7i54}
\int_{\mathbb M(X)}e^{i\langle f,\eta\rangle}\,d\mu(\eta)=
\exp\left[\int_{X\times \R^*}(e^{isf(x)}-1)\,d\sigma(x)\,d\nu(s)\right],\quad f\in B_0(X).
\end{equation}
Here $B_0(X)$ is the set of all bounded measurable functions on $X$ with compact support, and $\langle f,\eta\rangle:=\int_Xf\,d\eta$.

(ii) Let $\mu$ denote the distribution of a measure-valued L\'evy process on $X$. 
Let $\mathbb K(X)$ denote  the set of all signed discrete Radon measures on $X$, i.e.,
$$\mathbb K(X):=\left\{
\eta=\sum_i s_i\delta_{x_i}\in \mathbb M(X) \,\Big|\, s_i\ne0,\, x_i\in X
\right\}.$$
  Then $\mu(\mathbb K(X))=1$.

\end{proposition}

If the measure $\mu$ has Fourier transform \eqref{tyr7i54}, then $\nu$ is called the {\it L\'evy measure of $\mu$}.

\begin{remark}\label{utr7or} Assume in addition that $X$ is a smooth Riemannian manifold. Let $\nu$ be a measure on $\R^*$  which does not satisfy \eqref{utf7r} but satisfies the following weaker assumption: 
\begin{equation}\label{gf7}
\int_{\R^*}\min\{1,s^2\}\,d\nu(s)<\infty.\end{equation}
Then, instead of a measure-valued L\'evy process $\eta$, one may consider the associated 
{\it centered generalized stochastic process with independent values}, $\tilde \mu$, see \cite{GV}. In e.g.\ \cite{DOP}, $\tilde \mu$ is also called the  {\it centered L\'evy white noise measure}. 
Let us explain this in more detail.  
Let
$$\mathscr D(X)\subset L^2(X,\sigma)\subset\mathscr D'(X)$$ be the standard triple of spaces in which $\mathscr D(X)$ is the nuclear space of smooth, compactly supported functions on $X$ and
$\mathscr D'(X)$  is the dual space of $\mathscr D(X)$ with respect to zero space $L^2(X,\sigma)$, see e.g.\ \cite[Chap.~11, Sect.~1]{BUS}. Evidently, $\mathbb M(X)\subset \mathscr D'(X).$ For simplicity of notations, let us also assume that the following stronger assumption holds:
\begin{equation}\label{gf7tr7}
\int_{\R^*}s^2\,d\nu(s)<\infty,\end{equation}
 Then there exists a (centered) probability measure $\tilde\mu$ on $\mathscr D'(X)$ which has Fourier transform
$$\int_{\mathscr D'(X)}e^{i\langle f,\omega\rangle}\,d\tilde\mu(\omega)=
\exp\left[\int_{X\times \R^*}\big(e^{isf(x)}-1-isf(x)\big)\,d\sigma(x)\,d\nu(s)\right],\quad f\in \mathscr D(X).
$$
Here for  $\omega\in\mathscr D'(X)$ and $f\in\mathscr D(X)$, we denote by $\langle f,\omega\rangle$ the dual pairing between   $\omega$ and $f$. 
\end{remark}

By choosing in Proposition~\ref{vuytd7i} the measure $\nu$ to be $\delta_1$, we get the  {\it Poisson point process on $X$ with intensity measure $\sigma$}, whose distribution we denote by $\pi$. We have $\pi(\Gamma(X))$=1 and $\pi$ has Fourier transform
\begin{equation}\label{vtdry7r}
\int_{\Gamma(X)}e^{i\langle f,\gamma\rangle}\,d\pi(\gamma)=
\exp\left[\int_{X}(e^{if(x)}-1)\,d\sigma(x)\right],\quad f\in B_0(X).\end{equation}

 \begin{remark}\label{ur7o}
Let us equip $\R^*=\R\setminus\{0\}$ with a metric such that a set $\Lambda\subset\R^*$ is compact in $\R^*$ if and only if $\Lambda$ is compact in $\R$. In particular, for any compact set $\Lambda$ in $\R^*$ the distance (in $\R$) from $\Lambda$ to zero is strictly positive. We construct the following measurable injective (but not surjective) mapping
$$\mathbb K(X)\ni\eta=\sum_i s_i\delta_{x_i}\mapsto\mathscr R(\eta):=\{(x_i,s_i)\}\in\Gamma(X\times\R^*),$$
see \cite[Theorem~6.2]{HKPR}. As easily seen, the pushforward of the measure $\mu$ from Proposition~\ref{vuytd7i} under $\mathscr R$ is the Poisson point process on $X\times\R^*$ with intensity measure $\sigma\otimes\nu$.
\end{remark}

Now we will formulate two results of the classical probability theory whose analogs we want to derive in the free setting.

\begin{proposition}\label{fyer6} Let $\Lambda\in\mathscr B_0(X)$ with $\sigma(\Lambda)>0$. Let $(x_i)_{i=1}^\infty$ be a sequence of independent random points in $\Lambda$ with distribution $\sigma_\Lambda:=\frac1{\sigma(\Lambda)}\,\sigma$. Let $N$ be a random variable which has Poisson distribution with parameter $\sigma(\Lambda)$, and let $N$ be independent of $(x_i)_{i=1}^\infty$.
We define  $\pi_\Lambda$  as the distribution of the point process
$\sum_{i=1}^{N}\delta_{x_i}$.
Then $\pi_\Lambda$ is the restriction  to the set $\Lambda$ of the Poisson point process $\pi$, i.e., the pushforward of $\pi$ under the measurable mapping
$$\Gamma(X)\ni\gamma\mapsto\gamma\cap\Lambda\in\Gamma(\Lambda), $$
which coincides with the Poisson point process on $\Lambda$ with intensity $\sigma$.
\end{proposition}

\begin{remark}Quite often, Proposition~\ref{fyer6} is used as the definition of the Poisson point process $\pi$, i.e., one defines $\pi$ as the unique probability measure on $\Gamma(X)$ such that, for each $\Lambda\in\mathscr B_0(X)$, the restriction of $\pi$ to $\Lambda$ is equal to the measure $\pi_\Lambda$ from Proposition~\ref{fyer6}.
\end{remark}

\begin{proposition}[$N/V$-limit]\label{ye6i} Let $\Lambda^{(n)}\in\mathscr B_0(X)$, $n\in\mathbb N$, be such that $\Lambda^{(1)}\subset\Lambda^{(2)}\subset\Lambda^{(3)}\subset\cdots$ and $\bigcup_{n=1}^\infty\Lambda^{(n)}=X$, i.e., $\Lambda^{(n)}\nearrow  X$. Denote $V^{(n)}:=\sigma(\Lambda^{(n)})$, $V^{(n)}\to\infty$ as $n\to\infty$. Let $N^{(n)}\in\mathbb N$, $n\in\mathbb N$, be such that
\begin{equation}\label{nig8rf}
\lim_{n\to\infty}\frac{N^{(n)}}{V^{(n)}}=1.\end{equation}
(Note that the numbers $N^{(n)}$ are not random.)
For each $n\in\mathbb N$, we define $\rho^{(n)}$ as the distribution of the point process
\begin{equation}\label{gyit8p}
\sum_{i=1}^{N^{(n)}}\delta_{x_i^{(n)}},\end{equation} where $x_i^{(n)}$ are independent random points in $\Lambda^{(n)}$ with distribution $\sigma^{(n)}:=\frac{1}{V^{(n)}}\,\sigma$.
Then
$$\rho^{(n)}\to\pi\quad\text{weakly as }n\to\infty.$$

\end{proposition}

For a proof of Proposition \ref{ye6i}, see e.g.\ \cite{NZ}.

\begin{remark}\label{ftuyr6y7}
In view of Remark~\ref{ur7o}, the result of Proposition~\ref{ye6i} can be extended to a measure $\mu$ from Proposition~\ref{vuytd7i}.
More precisely, let $\Lambda^{(n)}\in\mathscr B_0(X)$, $\Lambda^{(n)}\nearrow X$, let $\Delta^{(n)}\in\mathscr B_0(\R^*)$,  $\Delta^{(n)}\nearrow \R^*$, let $V^{(n)}:=\sigma(\Lambda^{(n)})\nu(\Delta^{(n)})$, and let $N^{(n)}\in\mathbb N$ satisfy \eqref{nig8rf}. For each $n\in\mathbb N$, we define  $\rho^{(n)}$ as the distribution of the random measure
\begin{equation}\label{yr57i}
\sum_{i=1}^{N^{(n)}}s_i^{(n)}\delta_{x_i^{(n)}},\end{equation}
where $(x_i^{(n)},s_i^{(n)})$  
are independent random points in $\Lambda^{(n)}\times\Delta^{(n)}$   with distribution $\frac1{V^{(n)}}\sigma\otimes\nu$. Then $\rho^{(n)}\to\mu$ weakly as $n\to\infty$.

Furthermore, it can be shown that, if $X$ is a smooth Riemannian manifold and condition \eqref{gf7} is satisfied, then $\tilde\rho^{(n)}\to\tilde\mu$ weakly as $n\to\infty$. Here $\tilde\rho^{(n)}$ is the centered random measure $\rho^{(n)}$, i.e., the distribution of the random measure
$$\sum_{i=1}^{N^{(n)}}s_i^{(n)}\delta_{x_i^{(n)}}-\left(N^{(n)}\int_{\Delta^{(n)}} s\,d\nu^{(n)}(s)\right)\sigma^{(n)}.$$
\end{remark}

Let us also mention a characterization of a measure-valued L\'evy process through cumulants. Assume that the measure $\nu$ on $\R^*$ satisfies
\begin{equation}\label{yt5re7i}
\int_{\R^*} |s|^n\,d\nu(s)<\infty\quad\text{for all }n\in\mathbb N.\end{equation}
Then the measure $\mu$ from Proposition~\ref{vuytd7i}  has all moments finite, i.e., for any $f\in B_0(X)$ and $n\in\mathbb N$,
$$\int_{\mathbb K(X)}|\langle f,\eta\rangle|^n\,d\mu(\eta)<\infty.$$
According to the classical definition of cumulants,
for any $f_1,\dots,f_n\in B_0(X)$, the cumulant $C^{(n)}(\langle f_1,\cdot\rangle,\dots,\langle f_n,\cdot\rangle)$ is defined recurrently by the following formula, which connects the cumulants with the moments:
$$\int_{\mathbb K(X)}\langle f_1,\eta\rangle \langle f_2,\eta\rangle\dotsm\langle f_n,\eta\rangle\,d\mu(\eta)=\sum_{\theta\in\mathscr P(n)}\prod_{B\in\theta}C(B;\langle f_1,\cdot\rangle, \dots,\langle f_n,\cdot\rangle ).$$
Here $\mathscr P(n)$ denotes the collection of all set partitions of $\{1,2,\dots,n\}$ and for each $B=\{i_1,\dots,i_k\}\subset\{1,2,\dots,n\}$,
$$C\big(B;\langle f_1,\cdot\rangle, \dots,\langle f_n,\cdot\rangle\big ):=C^{(k)}\big(\langle f_{i_1},\cdot\rangle, \dots,\langle f_{i_k},\cdot\rangle\big).$$
A direct calculation shows that the cumulants of $\mu$ are given by
\begin{equation}\label{uytr75r}
C^{(n)}\big(\langle f_1,\cdot\rangle,\dots,\langle f_n,\cdot\rangle\big)=\int_{\R^*}s^n\,d\nu(s)\int_X f_1(x)f_2(x)\dotsm f_n(x)\,d\sigma(x),\quad n\in\mathbb N.\end{equation}
In particular, the cumulants of the Poisson point process $\pi$ are given by
\begin{equation}\label{uu7t}
C^{(n)}\big(\langle f_1,\cdot\rangle,\dots,\langle f_n,\cdot\rangle\big)=\int_X f_1(x)f_2(x)\dotsm f_n(x)\,d\sigma(x),\quad n\in\mathbb N.\end{equation}

\begin{remark}\label{uf7r7}
If the measure $\nu$ satisfies \eqref{yt5re7i} for $n\ge2$, then the first cumulant of the measure $\tilde\mu$ given in  Remark~\ref{utr7or} is equal
 to zero, while the cumulants $C^{(n)}$ for $n\ge2$ are still given by formula \eqref{uytr75r}.
\end{remark}

\section{Free independence and free L\'evy white noise}\label{niugt8i}

Let us recall the definition of free independence and free L\'evy white noise.

Let $\mathscr A$ be a von Neumann algebra of bounded linear operators acting in a separable Hilbert space, and let $\tau$ be a faithful normal trace which satisfies the condition $\tau(\mathbf 1)=1$.
The pair $(\mathscr A,\tau)$ is called a {\it $W^*$-probability space}.
An element $a\in\mathscr A$ is called a {\it noncommutative random variable}.

Let $\mathscr A_1,\dots,\mathscr A_n$ be subalgebras of $\mathscr A$. The algebras $\mathscr A_1,\dots,\mathscr A_n$ are called {\it freely independent\/} if, for any numbers $i_1,i_2,\dots,i_m\in\{1,\dots,n\}$ ($m\ge2$) such that $i_l\ne i_{l+1}$ for $l=1,\dots,m-1$, and for all
$a_l\in\mathscr A_{i_l}$ with $\tau(a_l)=0$, $l=1,\dots,m$, we have
$$\tau(a_1 a_2\dotsm a_m)=0.$$

Selfadjoint elements $a_1,\dots,a_n\in\mathscr A$ are called {\it freely independent} if the  algebras
$\mathscr A_1,\dots,\mathscr A_n$  that they respectively generate are freely independent.

The above definition of free independence is not very convenient for applications. Speicher \cite{Speicher94} (see also \cite{Speicher}) gave an equivalent definition based on the idea of free cumulants.

A set partition $\theta=\{B_1,\dots,B_k\}$ of $\{1,\dots,n\}$ is called {\it non-crossing\/} if there do not exist $B_i,B_j\in\theta$, $i\ne j$, for which the following inequalities hold: $x_1<y_1<x_2<y_2$ for some $x_1,x_2\in B_i$ and $y_1,y_2\in B_j$.   We denote by $\mathscr {NP}(n)$ the collection of all non-crossing set partitions of  $\{1,\dots,n\}$.

For $a_1,\dots,a_n\in\mathscr A$, the free cumulant $R^{(n)}(a_1,\dots,a_n)$ is defined recurrently by the following formula:
\begin{equation}\label{vgjuftuf}
\tau(a_1 a_2\dotsm a_n)=\sum_{\theta\in \mathscr {NP}(n)}\prod_{B\in\theta}R(B;a_1,\dots,a_n),\end{equation}
where for each $B=\{i_1,\dots,i_k\}\subset\{1,2,\dots,n\}$ with $i_1<i_2<\dots<i_k$,
\begin{equation}\label{gydfy6er}
 R(B;a_1,\dots,a_n):=R^{(k)}(a_{i_1},\dots,a_{i_k}).\end{equation}

\begin{theorem}[\cite{Speicher94}]\label{tyde6e6}
The selfadjoint elements $a_1,\dots,a_n\in\mathscr A$ are  freely independent if and only if, for each $k\ge2$ and any indices
$i_1,\dots,i_k\in\{1,\dots,n\}$ such that there exist $i_l\ne i_m$,  we have
$$R^{(k)}(a_{i_1},\dots,a_{i_k})=0.$$
\end{theorem}

Analogously to  \cite{BL}, we will now define a free L\'evy white noise. Let us first recall several definitions related to the   Fock space. Let $H$ be a real separable Hilbert space. We denote by $\mathscr F(H)$ the {\it (full) Fock space over $H$:}
$$\mathscr F(H):=\bigoplus_{n=0}^\infty H^{\otimes n}.$$
where $H^{\otimes0}:=\R$. The vector $\Omega=(1,0,0,\dots)$ is called the {\it vacuum}.

For $h\in H$, we define a {\it creation operator $a^+(h)$} as a bounded linear operator in $\mathscr F(H)$ satisfying
$$a^+(h)f^{(n)}=h\otimes f^{(n)},\quad f^{(n)}\in H^{\otimes n}.$$
The adjoint of $a^+(h)$ is an {\it annihilation operator}, which satisfies
$$a^-(h)f_1\otimes\dots\otimes f_n=(h,f_1)_H\, f_2\otimes\dots\otimes f_n,\quad f_1,\dots,f_n\in H.$$
For a bounded linear operator $L$ in $H$, we define a bounded linear operator $a^0(L)$ in $\mathscr F(H)$ by
$$a^0(L)\Omega:=0,\qquad a^0(L)f_1\otimes\dots\otimes f_n=
(Lf_1)\otimes f_2\otimes\dots\otimes f_n,\quad   f_1,\dots,f_n\in H.$$

Let $\nu$ be a measure on $\R^*$ which satisfies
\begin{equation}\label{gydyre6}
\nu\big(\{s\in\R^*\mid |s|>R\}\big)=0\quad \text{for some }R>0
\end{equation}
and
\begin{equation}\label{yre64i}
\int_{\R^*}|s|\,d\nu(s)<\infty.
\end{equation}

\begin{remark}
In fact, instead of requiring  \eqref{gydyre6} and \eqref{yre64i}, we could have supposed that $\nu$ satisfies \eqref{yt5re7i}. However, the operators appearing bellow would be, generally speaking, unbounded and we would have to  specify their domain. Still  the results of the paper which hold under the assumptions \eqref{gydyre6} and \eqref{yre64i} would remain true under the weaker assumption \eqref{yt5re7i}.
\end{remark}

Let $f\in B_0(X)$. We define $f\otimes\operatorname{id}:X\times\R^*\to\R$, $f\otimes\operatorname{id}(x,s):=f(x)s$.
By \eqref{gydyre6} and \eqref{yre64i},
$$ f\otimes\operatorname{id}\in L^2(X\times\R^*,\sigma\otimes\nu)\cap L^1(X\times\R^*,\sigma\otimes\nu)\cap L^\infty(X\times\R^*,\sigma\otimes\nu). $$
So, we define a selfadjoint bounded  linear operator
\begin{equation}\label{vgfuut7fr8}
 A(f):=a^+(f\otimes\operatorname{id})+ a^0(f\otimes\operatorname{id})+a^-(f\otimes\operatorname{id})+\int_X f\,d \sigma \int_{\R^*}s\,d\nu(s)\end{equation}
in $\mathscr F(L^2(X\times\R^*,\sigma\otimes\nu))$. When defining the operator $a^0(f\otimes\operatorname{id})$ we identified the function $f\otimes\operatorname{id}$ with the operator of multiplication by this function acting in $L^2(X\times\R^*,\sigma\otimes\nu)$.

Let $\mathscr A$ denote the real algebra generated by the operators $(A(f))_{f\in B_0(X)}$. We define a  trace $\tau$ on $\mathscr A$ by
\begin{equation}\label{gutr7o}
\tau (a)=(a\Omega,\Omega)_{\mathscr F(L^2(X\times\R^*,\sigma\otimes\nu))},\quad a\in\mathscr A.\end{equation}
A straightforward calculation shows that, for any $f_1,\dots,f_n\in B_0(X)$, the $n$th free cumulant is given by
\begin{equation}\label{git8i}
 R^{(n)}\big(A(f_1),\dots,A(f_n)\big)=\int_{\R^*}s^n\,d\nu(s)\int_X f_1(x)f_2(x)\dotsm f_n(x)\,d\sigma(x),\quad n\in\mathbb N,\end{equation}
compare with \eqref{uytr75r}. Hence, by Theorem~\ref{tyde6e6}, for any $f_1,\dots,f_n\in B_0(X)$ such that
$f_if_j=0$ $\sigma$-a.e.\ if $i\ne j$, the noncommutative random variables  $A(f_1),\dots,A(f_n)$ are freely independent. In particular, for any mutually disjoint sets $\Lambda_1,\dots,\Lambda_n\in\mathscr B_0(X)$, the noncommutative random variables $A(\Lambda_1),\dots,A(\Lambda_n)$ are freely independent. Here, for $\Lambda\in\mathscr B_0(X)$, we denote $A(\Lambda):=A(\chi_\Lambda)$, where $\chi_\Lambda$ is the indicator function of $\Lambda$.

Furthermore, for any $\Lambda_1,\Lambda_2\in\mathscr B_0(X)$ with $\sigma(\Lambda_1)=\sigma(\Lambda_2)$, we have
$$\tau(A(\Lambda_1)^n)=\tau(A(\Lambda_2)^n)\quad \text{for all }n\in\mathbb N.$$
Hence the selfadjoint operators $A(\Lambda_1)$ and $A(\Lambda_2)$ have the same spectral measure at the vacuum state $\Omega$. Thus, we can think of the family of operators $(A(f))_{f\in B_0(X)}$ as a {\it free L\'evy white noise with L\'evy measure $\nu$}.

In the case where $\nu=\delta_1$, the operators $A(f)$
act in the Fock space $\mathscr F(L^2(X,\sigma))$ and have the form
\begin{equation}\label{ctre6u4i5r}
A(f):=a^+(f)+ a^0(f)+a^-(f)+\int_X f\,d \sigma. \end{equation}
Their free cumulants are given by
\begin{equation}\label{yufr76r}
 R^{(n)}\big(A(f_1),\dots,A(f_n)\big)=\int_X f_1(x)f_2(x)\dotsm f_n(x)\,d\sigma(x),\quad n\in\mathbb N,\end{equation}
compare with \eqref{uu7t}.
The operators $(A(f))_{f\in B_0(X)}$ are called a {\it free Poisson process on $X$ with intensity measure $\sigma$.}

In the case where condition \eqref{gydyre6} is satisfied, but instead of \eqref{yre64i} the weaker condition \eqref{gf7tr7} holds, 
we still have
$$ f\otimes\operatorname{id}\in L^2(X\times\R^*,\sigma\otimes\nu)\cap  L^\infty(X\times\R^*,\sigma\otimes\nu). $$
So, we define selfadjoint operators
\begin{equation}\label{yufr87r}
 \tilde A(f):=a^+(f\otimes\operatorname{id})+ a^0(f\otimes\operatorname{id})+a^-(f\otimes\operatorname{id})\end{equation}
in $\mathscr F(L^2(X\times\R^*,\sigma\otimes\nu))$.
We similarly define $(\tilde {\mathscr A},\tau)$ and we get
$R^{(1)}(\tilde A(f))=0$ and
$$ R^{(n)}\big(\tilde A(f_1),\dots,\tilde A(f_n)\big)=\int_{\R^*}s^n\,d\nu(s)\int_X f_1(x)f_2(x)\dotsm f_n(x)\,d\sigma(x),\quad n\ge2.$$
Hence $(\tilde A(f))_{f\in B_0(X)}$ is a {\it centered free L\'evy white noise with L\'evy measure $\nu$}, compare with Remarks~\ref{utr7or} and~\ref{uf7r7}.

\begin{remark} Let $X=\R$ and let $d\sigma(x)=dx$ be the Lebesgue measure on $\R$. Let $\nu$ be a measure on $\R^*$ which satisfies  \eqref{gf7tr7}, \eqref{gydyre6}. For each $t\ge0$, we define $L(t):=\tilde A(\chi_{[0,t]})$.
Then $(L(t))_{t\ge0}$ is a centered free L\'evy process without free Gaussian part, see \cite{BNT}. 
\end{remark}

\section{Freely independent particles}\label{jkgt8yt}

In this section we will construct  systems of freely independent particles in a finite volume.

\subsection{$N$ freely independent particles}\label{utr75}

Assume $\Lambda\in\mathscr B_0(X)$. Let $P_\Lambda$ be a probability measure on $\Lambda$ and let $N\in\mathbb N$, $N\ge 2$. Assume that $x_1,\dots,x_N$ are independent random points in $\Lambda$ which have distribution $P_\Lambda$. We consider a point process $\gamma=\sum_{i=1}^N\delta_{x_i}$. We wish to introduce a counterpart of such a point process in which the particles $x_1,\dots,x_N$ are freely independent. This is done by generalizing the approach of Ben Arous and Kargin \cite{BAK} in the case $X=\R$.

 If we fix a function $f\in B_0(X)$ and integrate $f$ with respect to $\gamma$, we get the random variable
\begin{equation}\label{tye6u}
\sum_{i=1}^{N}f(x_{i}),\end{equation}
which is a sum of $N$ independent identically distributed random variables $f(x_{i})$. Since each $x_i$ has distribution $P_\Lambda$, we may think of each  random variable $f(x_i)$ as the operator  of multiplication by the function $f$ acting in the Hilbert space $L^2(\Lambda,P_\Lambda)$. We denote this operator by $M(\Lambda;f)$.
Thus, we may think of the classical random variable \eqref{tye6u} as the bounded linear operator
$$\sum_{i=1}^{N}M_i(\Lambda;f)$$
acting in the tensor product $L^2(\Lambda,P_\Lambda)^{\otimes N}$,  where the $i$th operator $M_i(\Lambda;f)$ is the $M(\Lambda;f)$ operator acting  in the $i$th  space of the tensor product $L^2(\Lambda,P_\Lambda)^{\otimes N}$.

To construct the free version, let us consider the  algebra generated by the (commutative) operators $(M(\Lambda;f))_{f\in B_0(X)}$ in $L^2(\Lambda,P_\Lambda)$. We define a trace $\tau_\Lambda$ on this algebra by setting, for any
$f_1,\dots,f_k\in B_0(X)$,
\begin{equation}\label{vytr57}
\tau_\Lambda\big(M(\Lambda;f_1)\dotsm M(\Lambda;f_k)\big):=
(M(\Lambda;f_1)\dotsm M(\Lambda;f_k)1,1)_{L^2(\Lambda, P_\Lambda)}
=\int_\Lambda f_1\dotsm f_k\,dP_\Lambda.\end{equation}

The next step is to construct the free product of $N$ copies of  this (commutative) $W^*$-probability space, see e.g.\ \cite{Maassen,NiSp}. For the reader's convenience, let us briefly recall this construction.

We define a Hilbert space
$$ H:=\left\{f\in L^2(\Lambda,P)\mid
\langle f\rangle=0\right\},$$
where $\langle f\rangle:=\int_\Lambda f\,dP_\Lambda$ and the Hilbert space $H$ is equipped with the scalar  product of $L^2(\Lambda,P_\Lambda)$.
Let $H_1,\dots,H_N$ denote $N$ copies of the space $H$. 
 We define a Hilbert space
$$\mathscr H(\Lambda,N):=\R\oplus\bigoplus_{\substack{k\in\mathbb N\\l_1,\dots,l_k\in\{1,\dots,N\}\\
l_j\ne l_{j+1},\,\, j=1,\dots,k-1}}H_{l_1}\otimes H_{l_2}\otimes\dots\otimes
H_{l_k}.$$
We set $\Psi_N:=(1,0,0,\dots)\in\mathscr H(\Lambda,N)$. For any $i\in\{1,\dots,N\}$ and $f\in B_0(X)$, we define a bounded linear operator $M_i(\Lambda;f)$ in $\mathscr H(\Lambda,N)$ by setting
$$ M_i(\Lambda;f)\Psi_N:=\langle f\rangle\Psi_N+[f]_i,$$
where $[f]_i:=f-\langle f\rangle\in H_i$, and for any $k\in\mathbb N$ and any $g_j\in H_{l_j}$, $j=1,\dots,k$ with $l_j\ne l_{j+1}$, we set: if  $i\ne l_1$,
$$M_i(\Lambda;f)g_1\otimes g_2\otimes\dots\otimes g_k:=
\langle f\rangle\, g_1\otimes g_2\otimes\dots\otimes g_k+
[f]_i\otimes g_1\otimes g_2\otimes\dots\otimes g_k
 $$
and if $i=l_1$,
 $$M_i(\Lambda;f)g_1\otimes g_2\otimes\dots\otimes g_k:=\langle fg_1\rangle\, g_2\otimes\dots\otimes g_k +[fg_1]_i\otimes g_2\otimes\dots\otimes g_k.$$

 We denote by $\mathscr M(\Lambda,N)$ the algebra generated by all the operators $M_i(\Lambda;f)$ in\linebreak  $\mathscr H(\Lambda,N)$, and by $\mathscr M_i(\Lambda,N)$ the subalgebra of $\mathscr M(\Lambda,N)$ that is generated by  the operators $M_i(\Lambda;f)$ with a fixed $i\in\{1,2,\dots,N\}$.
We define a trace $\tau_{\Lambda,N}$ on $\mathscr M(\Lambda,N)$ by
\begin{equation}\label{uyro67}
\tau_{\Lambda,N}(a):=(a\Psi_N,\Psi_N)_{\mathscr H(\Lambda,N)},\quad a\in\mathscr M(\Lambda,N).\end{equation}
Then the subalgebras $\mathscr M_i(\Lambda,N)$ with $i=1,\dots,N$
are freely independent in the $W^*$-probability space  $(\mathscr M(\Lambda,N),\tau_{\Lambda,N})$.
Furthermore, for each fixed $i$, we have
\begin{equation}\label{uftu7t6}
\tau_{\Lambda,N}\big(M_i(\Lambda;f_1)\dotsm M_i(\Lambda;f_k)\big)
=\tau_\Lambda\big(M(\Lambda;f_1)\dotsm M(\Lambda;f_k)\big).\end{equation}

For each $f\in B_0(X)$, we now set
\begin{equation}\label{bhyguyt}
A(\Lambda,N;f):=\sum_{i=1}^{N}M_i(\Lambda;f).\end{equation}
The operators $\big(A(\Lambda,N;f)\big)_{f\in B_0(X)}$ may be thought of as a system of $N$ freely independent particles in $\Lambda$ such that each particle has distribution $P_\Lambda$.
We denote by $\mathscr A(\Lambda,N)$ the subalgebra of $\mathscr M(\Lambda,N)$ that is generated by the operators $\big(A(\Lambda,N;f)\big)_{f\in B_0(X)}$. Thus, we have constructed the $W^*$-probability space $(\mathscr A(\Lambda,N),\tau_{\Lambda,N})$.

\subsection{Poisson-distributed random number of freely independent particles}\label{tre46u}

Next, let us consider the case where the number of particles, $N$, is random and has Poisson distribution with parameter $\alpha>0$.
To realize this situation, we proceed as follows. For a Hilbert space $\mathscr H$ and a constant $c>0$, we denote by $\mathscr Hc$ the Hilbert space which coincides with $\mathscr H$ as a set but the scalar product in $\mathscr Hc$ is equal to the scalar product in $\mathscr H$ times $c$.

We also denote $\mathscr H(\Lambda,0):=\R$, $\Psi_0:=1\in \mathscr H(\Lambda,0)$,
and $A(\Lambda,0;f):=0$ for all $f\in B_0(X)$. This corresponds to the case when there are no particles in $\Lambda$. In particular, $\Psi_0$ may be thought of as the indicator function of an empty set.

We define a Hilbert space
$$\mathscr H(\Lambda):=\bigoplus_{N=0}^\infty\mathscr H(\Lambda,N)\frac{\alpha^N}{N!}\,e^{-\alpha}.$$
We denote by $\mathscr H_{\mathrm{fin}}(\Lambda)$ the dense subset of $\mathscr H(\Lambda)$ which consists of all finite sequences $(F_0,F_1,\dots,F_k,0,0,\dots)$ with  $F_i\in \mathscr H(\Lambda,i)$, $i=1,\dots,k$,  $k\in\mathbb N$.

For each $f\in B_0(X)$, we define a (Hermitian) linear operator $A(\Lambda;f)$ in $\mathscr H(\Lambda)$ with domain  $\mathscr H_{\mathrm{fin}}(\Lambda)$ by
$$A(\Lambda;f)=\sum_{N=0}^\infty A(\Lambda,N;f),$$
where each operator $A(\Lambda,N;f)$ acts in $\mathscr H(\Lambda,N)$.
Evidently, each operator $A(\Lambda;f)$ maps $\mathscr H_{\mathrm{fin}}(\Lambda)$ into itself.

 We denote by $\mathscr A(\Lambda)$ the algebra generated by the operators $(A(\Lambda;f))_{f\in B_0(X)}$. A trace $\tau_\Lambda$ on $\mathscr A(\Lambda)$ is given by
\begin{equation}\label{vufr87}
\tau_\Lambda(a):=\sum_{N=0}^\infty (a\Psi_N,\Psi_N)_{\mathscr H(\Lambda,N)}\,\frac{\alpha^N}{N!}\,e^{-\alpha},\quad a\in \mathscr A(\Lambda).\end{equation}
As easily seen, the series in formula \eqref{vufr87} indeed converges for each $a\in \mathscr A(\Lambda)$.

\begin{remark}
As the operators $A(\Lambda;f)$ are unbounded, strictly speaking $(\mathscr A(\Lambda),\tau_\Lambda)$ is not a $W^*$-probability space. We will call it a {\it noncommutative probability space}.
\end{remark}

\subsection{Discrete measures with freely independent atoms\\ and weights}\label{bjftfryuf}

Let $\Delta\in\mathscr B_0(\R^*)$ and let $P_{\Lambda\times\Delta}$ be a probability measure on $\Lambda\times\Delta$. Let $N\in\mathbb N$, $N\ge2$. Assume that $(x_1,s_1),\dots,(x_N,s_N)$ are independent random points in $\Lambda\times\Delta$ which have distribution $P_{\Lambda\times\Delta}$. We consider a random discrete measure which has atoms at the points $x_i$ and weights $s_i$:
$\eta=\sum_{i=1}^N s_i\delta_{x_i}$.
Analogously to subsec.~\ref{utr75}, we may now easily introduce a free analog of this random  measure.  Indeed, integrating  a function $f\in B_0(X)$ with respect to $\eta$, we get a random variable
\begin{equation}\label{ghf7i}
\sum_{i=1}^N s_if(x_i),\end{equation}
 which is a sum of $N$ independent identically distributed random variables $s_if(x_i)$.

 So, starting with the Hilbert space $L^2(\Lambda\times\Delta,P_{\Lambda\times\Delta})$ instead of $L^2(\Lambda,P_\Lambda)$,
we define a Hilbert space $\mathscr H(\Lambda\times\Delta,N)$ analogously to $\mathscr H(\Lambda,N)$.
Using the function $f\otimes\operatorname{id}$ instead of $f$, we define bounded linear operators $M_i(\Lambda\times\Delta;f)$ in
 $\mathscr H(\Lambda\times\Delta,N)$.
  Analogously to subsec.~\ref{utr75}, we define an algebra $\mathscr M(\Lambda\times\Delta,N)$,
a trace $\tau_{\Lambda\times\Delta}$ on it,  and subalgebras
$\mathscr M_i(\Lambda\times\Delta,N)$, which are freely independent.
Finally, we define operators $A(\Lambda\times\Delta,N;f)$ and the corresponding subalgebra $\mathscr A(\Lambda\times\Delta,N)$.
(We have used obvious notations.)
 The $W^*$-probability space $\big(\mathscr A(\Lambda\times\Delta,N),\tau_{\Lambda\times\Delta,N}\big)$ is the required free analog of the random measure \eqref{ghf7i}.

Let us also note that, analogously to subsec.~\ref{tre46u}, we may also construct a  noncommutative probability space corresponding to the case where the number of particles, $N$, in $\mathscr A(\Lambda\times\Delta,N)$ is random and has Poisson distribution with parameter $\alpha$. We denote the corresponding operators by $A(\Lambda\times\Delta;f)$, and the  noncommutative probability space by $(\mathscr A(\Lambda\times\Delta),\tau_{\Lambda\times\Delta})$.

\section{Approximations}\label{lngity8t}

In this section, we will formulate the main results of the paper, the proofs will be given in Section~\ref{yufu8} below.

Our first result is the free counterpart of Proposition~\ref{ye6i}.

\begin{theorem}[$N/V$ limit for the free Poisson process]\label{tyr7}
Let $\Lambda^{(n)}\in\mathscr B_0(X)$, $n\in\mathbb N$, and  $\Lambda^{(n)}\nearrow X$. Denote $V^{(n)}:=\sigma(\Lambda^{(n)})$, $V^{(n)}\to\infty$ as $n\to\infty$. Let $N^{(n)}\in\mathbb N$, $n\in\mathbb N$, be such that \eqref{nig8rf} holds.
For each $n\in\mathbb N$, consider the $W^*$-probability space
\begin{equation}\label{futfr75r7i}
\big(\mathscr A(\Lambda^{(n)},N^{(n)}),\tau_{\Lambda^{(n)},N^{(n)}}\big),
\end{equation}
constructed in subsec.~\ref{utr75}, which describes  $N^{(n)}$ freely independent particles in $\Lambda^{(n)}$ with distribution $\sigma^{(n)}=\frac1{V^{(n)}}\sigma$.
Then, as $n\to\infty$, the $W^*$-probability space \eqref{futfr75r7i}  converges to the $W^*$-probability space
 $(\mathscr A,\tau)$
 of the free Poisson process on $X$ with intensity $\sigma$. The convergence is in the sense of moments, i.e., for any $f_1,\dots,f_k\in B_0(X)$,
\begin{equation}\label{ufu8u}
\lim_{n\to\infty}\tau_{\Lambda^{(n)},N^{(n)}} \big( A(\Lambda^{(n)},N^{(n)};f_1)\dotsm A(\Lambda^{(n)},N^{(n)};f_k)\big)
=
\tau(A(f_1)\dotsm A(f_k)).\end{equation}
Here, for each $f\in B_0(X)$, $A(f)$ is the operator in $\mathscr F(L^2(X,\sigma))$ defined by \eqref{ctre6u4i5r}, and $\tau$ is given by formula \eqref{gutr7o} with $\nu=\delta_1$.
\end{theorem}

\begin{remark}
In the case $X=\R$, Theorem 2 in \cite{BAK} implies that the $W^*$-probability space  \eqref{futfr75r7i}  converges weakly  to the $W^*$-probability space  $(\mathscr A,\tau)$, i.e., for each $f\in B_0(X)$ and $k\in\mathbb N$,
\begin{equation}\label{lyi79ty}
\lim_{n\to\infty}\tau_{\Lambda^{(n)},N^{(n)}} \big( A(\Lambda^{(n)},N^{(n)};f)^k\big)=\tau(A(f)^k).\end{equation}
However, in the noncommutative setting, formula \eqref{lyi79ty} does not imply the stronger statement \eqref{ufu8u}.
\end{remark}

The following statement is a free analog of Remark~\ref{ftuyr6y7}.

\begin{corollary}[$N/V$ limit for the free L\'evy white noise]\label{nhiut87t866t9}
Assume that a measure $\nu$ on $\mathbb R^*$ satisfies \eqref{gydyre6} and \eqref{yre64i}.
Let $\Lambda^{(n)}\in\mathscr B_0(X)$, $\Lambda^{(n)}\nearrow X$, let $\Delta^{(n)}\in\mathscr B_0(\R^*)$,  $\Delta^{(n)}\nearrow \R^*$, let $V^{(n)}:=\sigma(\Lambda^{(n)})\nu(\Delta^{(n)})$, and let $N^{(n)}\in\mathbb N$ satisfy \eqref{nig8rf}. For each $n\in\mathbb N$, consider the $W^*$-probability space
\begin{equation}\label{gyutr7}\big(\mathscr A(\Lambda^{(n)}\times\Delta^{(n)},N^{(n)}),\tau_{\Lambda^{(n)}\times\Delta^{(n)},N^{(n)}}\big),\end{equation}
constructed in subsec.~\ref{bjftfryuf}, which describes  the sum of $N^{(n)}$ freely independent random measures $s_i\delta_{x_i}$, where each $(s_i,x_i)$ is a random point in $\Lambda^{(n)}\times\Delta^{(n)}$ with distribution
$\frac1{V^{(n)}}\sigma\otimes\nu$.
Then, as $n\to\infty$, the $W^*$-probability space \eqref{gyutr7}  converges
to the $W^*$-probability space $(\mathscr A,\tau)$
which describes the free L\'evy white noise with L\'evy measure $\nu$. The convergence is in the sense of moments, i.e., for any $f_1,\dots,f_k\in B_0(X)$,
\begin{align}&\lim_{n\to\infty}\tau_{\Lambda^{(n)}\times\Delta^{(n)},N^{(n)}} \big( A(\Lambda^{(n)}\times\Delta^{(n)},N^{(n)};f_1)\dotsm A(\Lambda^{(n)}\times\Delta^{(n)},N^{(n)};f_k)\big)\notag\\
&\quad=
\tau(A(f_1)\dotsm A(f_k)).\label{kjlogy8ty8}\end{align}
Here, for each $f\in B_0(X)$, $A(f)$ is the operator in $\mathscr F(L^2(X\times\R^*,\sigma\otimes\nu))$ defined by \eqref{vgfuut7fr8}, and $\tau$ is given by \eqref{gutr7o}.

If instead of \eqref{yre64i}, the measure $\nu$ satisfies the weaker assumption \eqref{gf7tr7}, we still have the following approximation of the centered free L\'evy white noise with L\'evy measure $\nu$: for any $f_1,\dots,f_k\in B_0(X)$,
\begin{align}
&\lim_{n\to\infty}\tau_{\Lambda^{(n)}\times\Delta^{(n)},N^{(n)}} \big( \tilde A(\Lambda^{(n)}\times\Delta^{(n)},N^{(n)};f_1)\dotsm \tilde A(\Lambda^{(n)}\times\Delta^{(n)},N^{(n)};f_k)\big)\notag\\
&\quad=
\tau(\tilde A(f_1)\dotsm \tilde A(f_k)).\label{gjuf7tut7}\end{align}
Here
$$ \tilde A(\Lambda\times\Delta,N;f):= A(\Lambda\times\Delta,N;f)-\tau_{\Lambda\times\Delta,N}(A(\Lambda\times\Delta,N;f))$$
and $\tilde A(f)$ is defined by \eqref{yufr87r}.
\end{corollary}

According to Proposition~\ref{fyer6}, in  the classical case the measure $\pi_\Lambda$ is precisely the restriction  to $\Lambda$ of the Poisson point process $\pi$. As will be seen from the proof of Theorem~\ref{hgf7utr76tr6} below,
 a direct analog of this fact is not true in the free case, i.e., the free counterpart of $\pi_\Lambda$ is not the restriction to $\Lambda$  of the free Poisson process. 
 
\begin{theorem}[approximation of the free Poisson process]\label{hgf7utr76tr6}
Let $\Lambda^{(n)}\in\mathscr B_0(X)$, $n\in\mathbb N$, and  $\Lambda^{(n)}\nearrow X$, $\sigma(\Lambda^{(n)})\to\infty$ as $n\to\infty$. For each $n\in\mathbb N$, consider the noncommutative probability space $(\mathscr A(\Lambda^{(n)}),\tau_{\Lambda^{(n)}})$, constructed in subsec.~\ref{tre46u}, which describes  $N^{(n)}$ freely independent particles in $\Lambda^{(n)}$ with distribution $\sigma^{(n)}$, where $N^{(n)}$ is a random number which has Poisson distribution with parameter $\sigma(\Lambda^{(n)})$.
Then, as $n\to\infty$, the noncommutative probability space $(\mathscr A(\Lambda^{(n)}),\tau_{\Lambda^{(n)}})$ converges to the $W^*$-probability space
 $(\mathscr A,\tau)$
 of the free Poisson process on $X$ with intensity $\sigma$. The convergence is in the sense of moments, i.e., for any $f_1,\dots,f_k\in B_0(X)$,
\begin{equation}\label{kigt8uy6t}
\lim_{n\to\infty}\tau_{\Lambda^{(n)}} \big( A(\Lambda^{(n)};f_1)\dotsm A(\Lambda^{(n)};f_k)\big)
=
\tau(A(f_1)\dotsm A(f_k)).\end{equation}
Here, the operators $A(f_i)$ and the state $\tau$ are as in Theorem~\ref{tyr7}.
\end{theorem}

The corresponding result also holds for a free L\'evy white noise.

\begin{corollary}[approximation of the free L\'evy white noise]\label{kigy8u}
Assume that a measure $\nu$ on $\mathbb R^*$ satisfies \eqref{gydyre6} and \eqref{yre64i}.
Let $\Lambda^{(n)}\in\mathscr B_0(X)$, $\Lambda^{(n)}\nearrow X$, let $\Delta^{(n)}\in\mathscr B_0(\R^*)$,  $\Delta^{(n)}\nearrow \R^*$, and let $V^{(n)}:=\sigma(\Lambda^{(n)})\nu(\Delta^{(n)})$.
For each $n\in\mathbb N$, consider the noncommutative probability space
\begin{equation}\label{bhjgfyu8}
(\mathscr A(\Lambda^{(n)}\times\Delta^{(n)}),\tau_{\Lambda^{(n)}\times\Delta^{(n)}}),\end{equation}
 constructed in subsec.~\ref{bjftfryuf}, which describes  a sum of $N^{(n)}$ freely independent random measures $s_i\delta_{x_i}$, where each $(s_i,x_i)$ is a random point in $\Lambda^{(n)}\times\Delta^{(n)}$ with distribution
$\frac1{V^{(n)}}\sigma\otimes\nu$, and $N^{(n)}$ is a random number which has Poisson distribution with parameter $V^{(n)}$.
Then, as $n\to\infty$, the noncommutative probability space \eqref{bhjgfyu8} converges (in the sense of moments) to the $W^*$-probability space $(\mathscr A,\tau)$
which describes the free L\'evy white noise with L\'evy measure $\nu$.

If instead of \eqref{yre64i}, the measure $\nu$ satisfies the weaker assumption \eqref{gf7tr7}, we still have the following approximation of the centered free L\'evy white noise with L\'evy measure $\nu$: for any $f_1,\dots,f_k\in B_0(X)$,
\begin{equation}\label{hjuft}
\lim_{n\to\infty}\tau_{\Lambda^{(n)}\times\Delta^{(n)}} \big( \tilde A(\Lambda^{(n)}\times\Delta^{(n)};f_1)\dotsm \tilde A(\Lambda^{(n)}\times\Delta^{(n)};f_k)\big)\\
=
\tau(\tilde A(f_1)\dotsm \tilde A(f_k)).\end{equation}
Here $\tilde A(f)$ is defined by \eqref{yufr87r}, and 
$$\tilde A(\Lambda^{(n)}\times\Delta^{(n)};f)=\sum_{N=0}^\infty \tilde A(\Lambda^{(n)}\times\Delta^{(n)},N;f),$$ where the operator
$$\tilde A(\Lambda^{(n)}\times\Delta^{(n)},N;f):=A(\Lambda^{(n)}\times\Delta^{(n)},N;f)-
\tau_{\Lambda^{(n)}\times\Delta^{(n)},N}\big(A(\Lambda^{(n)}\times\Delta^{(n)},N;f)\big)$$
acts in $\mathscr H(\Lambda^{(n)}\times\Delta^{(n)},N)$.

\end{corollary}

\begin{remark}
The results of this section may be seen as the free analogs of equivalence of ensembles for the ideal gas, cf.\  \cite{Ellis}.

\end{remark}

\section{Proofs}\label{yufu8}

The results presented in the lemma below are well known and can be easily derived from the definition of a free cumulant.

\begin{lemma}\label{klnhuighi}
Let $\mathscr A$ be an algebra and let $\tau$ be a trace on  $\mathscr A$. Let the free cumulants on $\mathscr A$ be defined through \eqref{vgjuftuf}.
 The following statements hold.

(i) For each $n\in\mathbb N$, there exist $c_\theta\in\mathbb Z$ with  $\theta\in\mathscr{NP}(n)$ such that 
 \begin{equation}\label{jkgtu8t}
 R^{(n)}(a_1,\dots,a_n)=\sum_{\theta\in\mathscr{NP}(n)}\prod_{B\in\theta}c_\theta\,\tau(B,a_1,\dots,a_n), \quad a_1,\dots,a_n\in\mathscr A.\end{equation}
Here, for each $B=\{i_1,\dots,i_k\}\subset\{1,2,\dots,n\}$ with $i_1<i_2<\dots<i_k$,
$$\tau(B;a_1,\dots,a_n):=\tau(a_{i_1}\dotsm a_{i_k}).$$
Furthermore, for the partition $\theta$ which has only one element, $\{1,2,\dots,n\}$, we have  $c_\theta=1$.

(ii) For any $a\in\mathscr A$, denote $\tilde a:=a-\tau(a)$. Then $R^{(1)}(\tilde a)=0$, any for any $a_1,\dots,a_n\in\mathscr A$ with $n\ge2$,
\begin{equation}\label{kjgiy}
R^{(n)}(\tilde a_1,\dots,\tilde a_n)=R^{(n)}(a_1,\dots,a_n).\end{equation}

\end{lemma}

\begin{proof}[Proof of Theorem \ref{tyr7}]
By \eqref{vgjuftuf} and \eqref{jkgtu8t},  formula \eqref{ufu8u} is equivalent to the following statement: for any $f_1,\dots,f_k\in B_0(X)$, $k\in\mathbb N$,
\begin{equation}\label{jkohiuy}
\lim_{n\to\infty} R^{(k)}_{\Lambda^{(n)},N^{(n)}}\big(
A(\Lambda^{(n)},N^{(n)};f_1),\dots,A(\Lambda^{(n)},N^{(n)};f_k)
\big)=R^{(k)}(A(f_1),\dots,A(f_k))\end{equation}
(we used obvious notations). By Theorem~\ref{tyde6e6} and  formulas \eqref{uftu7t6}, \eqref{bhyguyt}, we get
\begin{align}
&R^{(k)}_{\Lambda^{(n)},N^{(n)}}\big(A(\Lambda^{(n)},N^{(n)};f_1),\dots,A(\Lambda^{(n)},N^{(n)};f_k)\big)\notag\\
&\quad=\sum_{i_1=1}^{N^{(n)}}\dotsm\sum_{i_k=1}^{N^{(n)}}R^{(k)}_{\Lambda^{(n)},N^{(n)}}\big(
M_{i_1}(\Lambda^{(n)};f_1),\dots, M_{i_k}(\Lambda^{(n)};f_k)\big)\notag\\
&\quad=\sum_{i=1}^{N^{(n)}}R^{(k)}_{\Lambda^{(n)},N^{(n)}}\big(
M_{i}(\Lambda^{(n)};f_1),\dots, M_{i}(\Lambda^{(n)};f_k)\big)\notag\\
&\quad=N^{(n)}R_{\Lambda^{(n)}}^{(k)}\big(
M(\Lambda^{(n)};f_1),\dots,M(\Lambda^{(n)};f_k)\big).\label{uuft7tr}
\end{align}
Here, for $f\in B_0(X)$, $M(\Lambda^{(n)};f)$ is the operator of multiplication by the function $f$ in $L^2(\Lambda^{(n)},\sigma^{(n)})$, the (commutative) algebra generated by these operators is equipped with the trace $\tau_{\Lambda^{(n)}}$, see \eqref{vytr57}, and $R_{\Lambda^{(n)}}^{(k)}$, $k\in\mathbb N$ are the corresponding free cumulants.

By Lemma~\ref{klnhuighi}, (i) and the formula \eqref{vytr57} with $\Lambda=\Lambda^{(n)}$ and $P_\Lambda=\sigma^{(n)}=\frac1{V^{(n)}}\sigma$, there exist  bounded sequences of real numbers, $(c^{(n)}_2)_{n=1}^\infty,\dots,(c^{(n)}_k)_{n=1}^\infty$,  which depend on $f_1,\dots,f_k$ but are independent of $N^{(n)}$, such that
\begin{equation}\label{jig8u}
R_{\Lambda^{(n)}}^{(k)}\big(
M(\Lambda^{(n)};f_1),\dots,M(\Lambda^{(n)};f_k)\big)=\frac1{V^{(n)}}\int_{\Lambda^{(n)}} f_1(x)\dotsm f_k(x)\,d\sigma(x)+\sum_{j=2}^k
\frac{c^{(n)}_j}{(V^{(n)})^j}.
\end{equation}
By \eqref{uuft7tr} and \eqref{jig8u}, we get
$$\lim_{n\to\infty} R^{(k)}_{\Lambda^{(n)},N^{(n)}}\big(A(\Lambda^{(n)},N^{(n)};f_1),\dots,A(\Lambda^{(n)},N^{(n)};f_k)\big)=
\int_X f_1(x)\dotsm f_k(x)\,d\sigma(x),$$
which, in view of \eqref{yufr76r}, gives \eqref{jkohiuy}.
\end{proof}

\begin{proof}[Proof of Corollary \ref{nhiut87t866t9}]
The proof of formula \eqref{kjlogy8ty8} is similar to the proof of \eqref{ufu8u}. Indeed, analogously to \eqref{jig8u}, we get 
\begin{align}
&R_{\Lambda^{(n)}\times\Delta^{(n)}}^{(k)}\big(
M(\Lambda^{(n)}\times \Delta^{(n)};f_1),\dots,M(\Lambda^{(n)}\times\Delta^{(n)};f_k)\big)\notag\\
&\quad=\frac1{V^{(n)}}\int_{\Lambda^{(n)}} f_1(x)\dotsm f_k(x)\,d\sigma(x)
\int_{\Delta^{(n)}}s^k\,d\nu(s)
+\sum_{j=2}^k
\frac{c_j^{(n)}}{(V^{(n)})^j},\label{ytdr6t}
\end{align}
which implies
\begin{align}
&\lim_{n\to\infty} R^{(k)}_{\Lambda^{(n)}\times\Delta^{(n)},N^{(n)}}\big(A(\Lambda^{(n)}\times\Delta^{(n)},N^{(n)};f_1),
\dots,A(\Lambda^{(n)}\times\Delta^{(n)},N^{(n)};f_k)\big)\notag\\
&\quad =
\int_X f_1(x)\dotsm f_k(x)\,d\sigma(x)\int_{\R^*}s^k\,d\nu(s).\label{gufr7yr}\end{align}
By \eqref{git8i}, formula \eqref{kjlogy8ty8} holds.

Formula \eqref{gjuf7tut7} is proven as follows. In view of Lemma~\ref{klnhuighi}, (ii), it suffices to show that, under the assumptions \eqref{gydyre6} and \eqref{gf7tr7}, formula \eqref{gufr7yr} holds for $k\ge2$. Note that
$$\sup_{n\in\mathbb N}\left(\frac1{\nu(\Delta^{(n)})}\int_{\Delta^{(n)}}s\,d\nu(s)\right)<\infty.$$
Using this, one can easily show that the constants $c_j^{(n)}$ from \eqref{ytdr6t} satisfy, for $j\ge2$, 
$$\sup_{n\in\mathbb N}\frac{|c_j^{(n)}|}{\nu(\Delta^{(n)})^j}<\infty.$$
Therefore,
$$\frac{|c_j^{(n)}|}{(V^{(n)})^j}=\frac{|c_j^{(n)}|}{\nu(\Delta^{(n)})^j\,\sigma(\Lambda^{(n)})^j}\to0\quad\text{as }n\to\infty,$$
which  implies \eqref{gufr7yr} for $k\ge2$.
\end{proof}

\begin{proof}[Proof of Theorem \ref{hgf7utr76tr6}]
For $n\in\mathbb N$, denote $V^{(n)}:=\sigma(\Lambda^{(n)})$.
Let $f_1,\dots,f_k\in B_0(X)$.  We have, by \eqref{vgjuftuf}, \eqref{vufr87} and \eqref{uuft7tr},
\begin{align}
&\tau_{\Lambda^{(n)}}\big(
A(\Lambda^{(n)};f_1)\dotsm A(\Lambda^{(n)};f_k)\big)\notag\\
&\quad=\sum_{N=1}^\infty \tau_{\Lambda^{(n)},N}\big(
A(\Lambda^{(n)},N;f_1)\dotsm A(\Lambda^{(n)},N;f_k)\big)\,\frac{(V^{(n)})^N}{N!}e^{-V^{(n)}}\notag\\
&\quad=e^{-V^{(n)}}\sum_{N=1}^\infty \frac{(V^{(n)})^N}{N!} \sum_{\theta\in\mathscr{NP}(k)}\prod_{B\in\theta}R_{\Lambda^{(n)},N}\big(B;
A(\Lambda^{(n)},N;f_1),\dots, A(\Lambda^{(n)},N;f_k)\big)\notag\\
&\quad=e^{-V^{(n)}}\sum_{i=1}^k \sum_{N=1}^\infty \frac{(V^{(n)})^N N^{i-1}}{(N-1)!}
\sum_{\substack{\theta\in\mathscr{NP}(k)\\ |\theta|=i}} R_{\Lambda^{(n)}}\big(\theta;
M(\Lambda^{(n)};f_1),\dots, M(\Lambda^{(n)};f_k)\big).\label{nbfgu}
\end{align}
Here $|\theta|$ denotes the number of sets in the partition $\theta$ and 
\begin{equation*}
 R_{\Lambda^{(n)}}\big(\theta;
M(\Lambda^{(n)};f_1),\dots, M(\Lambda^{(n)};f_k)\big):=\prod_{B\in\theta}R_{\Lambda^{(n)}}\big(B;
M(\Lambda^{(n)};f_1),\dots, M(\Lambda^{(n)};f_k)\big). \end{equation*}

 Let us consider the term in the sum $\sum_{i=1}^k$ in \eqref{nbfgu} which corresponds to $i=1$. 
 The only partition $\theta\in\mathscr{NP}(k)$ with $|\theta|=1$ is $\theta=\big\{\{1,\dots,k\}\big\}$.
 Hence, by \eqref{vgjuftuf} and \eqref{vytr57}, this term is equal to 
\begin{align}
& e^{-V^{(n)}}\sum_{N=1}^\infty \frac{(V^{(n)})^N }{(N-1)!}\, R^{(k)}_{\Lambda^{(n)}}\big(
M(\Lambda^{(n)};f_1),\dots, M(\Lambda^{(n)};f_k)\big)\notag\\
&\quad=  V^{(n)}\bigg(\tau_{\Lambda^{(n)}}\big(
M(\Lambda^{(n)};f_1)\dotsm M(\Lambda^{(n)};f_k)\big)\notag\\
&\qquad- \sum_{\substack{\theta\in\mathscr{NP}(k)\\|\theta|\ge2}}
R_{\Lambda^{(n)}}\big(\theta;M(\Lambda^{(n)};f_1),\dots, M(\Lambda^{(n)};f_k)\big)\bigg)\notag\\
&\quad=\int_{\Lambda^{(n)}} f_1(x)\dotsm f_k(x)\,d\sigma(x)\notag\\
&\qquad-V^{(n)}\sum_{\substack{\theta\in\mathscr{NP}(k)\\|\theta|\ge2}}
R_{\Lambda^{(n)}}\big(\theta;
M(\Lambda^{(n)};f_1),\dots, M(\Lambda^{(n)};f_k)\big).\label{nbjghi}
\end{align}
By   \eqref{nbfgu} and \eqref{nbjghi}, we get
\begin{align}
&\tau_{\Lambda^{(n)}}\big(
A(\Lambda^{(n)};f_1)\dotsm A(\Lambda^{(n)};f_1)\big)
=\int_{\Lambda^{(n)}} f_1(x)\dotsm f_k(x)\,d\sigma(x)\notag\\
&\quad+e^{-V^{(n)}}\sum_{i=2}^k\sum_{N=2}^\infty \frac{(V^{(n)})^N (N^{i-1}-1)}{(N-1)!}
 \sum_{\substack{\theta\in\mathscr{NP}(k)\\ |\theta|=i}} R_{\Lambda^{(n)}}\big(\theta;
M(\Lambda^{(n)};f_1),\dots, M(\Lambda^{(n)};f_k)\big).\label{hjgcukF}
\end{align}

We denote, for each partition $\theta\in\mathscr{NP}(k)$ and $n\in\mathbb N$,
\begin{equation}\label{hjfut}
 I^{(n)}(\theta; f_1,\dots,f_k):=\prod_{B\in\theta} I^{(n)}(B; f_1,\dots,f_k),\end{equation}
where for $B=\{i_1,\dots,i_l\}\in\theta$
\begin{equation}\label{ioy89}
I^{(n)}(B; f_1,\dots,f_k)=\int_{\Lambda^{(n)}} f_{i_1}(x)\dotsm f_{i_l}(x)\,d\sigma(x).\end{equation}
Consider any $\theta\in\mathscr{NP}(k)$ with $|\theta|=2$. We have
\begin{align}
&R_{\Lambda^{(n)}}\big(\theta;
M(\Lambda^{(n)};f_1),\dots, M(\Lambda^{(n)};f_k)\big)=\frac1{(V^{(n)})^2}\,I^{(n)}(\theta; f_1,\dots,f_k)\notag\\
&\qquad- \sum_{\substack{\pi\in\mathscr{NP}(k)\\ |\pi|\ge3,\, \pi\le\theta}}R_{\Lambda^{(n)}}\big(\pi;
M(\Lambda^{(n)};f_1),\dots, M(\Lambda^{(n)};f_k)\big).\label{iyt867tr}
\end{align}
Here $\pi\le\theta$ denotes that the partition $\pi$ is finer than the partition $\theta$, i.e., each element of $\pi$ is a subset of some element of $\theta$.

Recall that, for $i,j\in\mathbb N$, $i\ge j$, the {\it Stirling number of the second kind}, $\operatorname{S}(i,j)$, denotes  the number of ways to partition a set of $i$ elements into $j$ nonempty subsets.
Let $\pi\in\mathscr{NP}(k)$ with $|\pi|\ge3$. For $l\in\{1,\dots,k\}$ with $l\le |\pi|$, we denote by $\operatorname{NS}
(\pi,l)$ the number of non-crossing partitions $\theta\in \mathscr{NP}(k)$ such that $|\theta|=l$ and $\pi\le\theta$.
 Note that the number of all  partitions $\theta$ of $\{1,\dots,k\}$ such that $|\theta|=l$ and $\pi\le\theta$ is equal to $\operatorname{S}(|\pi|,l)$. Hence, $\operatorname{NS}(\pi,l)\le \operatorname{S}(|\pi|,l)$.

By \eqref{hjgcukF}--\eqref{iyt867tr}, we get
\begin{align}
&\tau_{\Lambda^{(n)}}\big(
A(\Lambda^{(n)};f_1)\dotsm A(\Lambda^{(n)};f_1)\big)
-\sum_{i=1,2}\,\sum_{\substack{\theta\in\mathscr{NP}(k)\\ |\theta|=i}}I^{(n)}(\theta;f_1,\dots, f_k),\notag\\
&\quad=e^{-V^{(n)}}\sum_{i=3}^k\sum_{N=2}^\infty \frac{(V^{(n)})^N }{(N-1)!}
\sum_{\substack{\theta\in\mathscr{NP}(k)\\ |\theta|=i}} R_{\Lambda^{(n)}}\big(\theta;
M(\Lambda^{(n)};f_1),\dots, M(\Lambda^{(n)};f_k)\big)\notag\\
&\qquad\times \big(N^{i-1}-1-(N-1)\operatorname{NS}(\theta,2)\big)\notag\\
&\quad=e^{-V^{(n)}}\sum_{i=3}^k\sum_{N=2}^\infty \frac{(V^{(n)})^N }{(N-1)!}\, \big(N^{i-1}-1-(N-1)\operatorname{S}(i,2)\big)
\notag\\
&\qquad\times \sum_{\substack{\theta\in\mathscr{NP}(k)\\ |\theta|=i}} R_{\Lambda^{(n)}}\big(\theta;
M(\Lambda^{(n)};f_1),\dots, M(\Lambda^{(n)};f_k)\big)\notag\\
&\qquad+(V^{(n)})^2
\sum_{i=3}^k\sum_{\substack{\theta\in\mathscr{NP}(k)\\ |\theta|=i}} R_{\Lambda^{(n)}}\big(\theta;
M(\Lambda^{(n)};f_1),\dots, M(\Lambda^{(n)};f_k)\big)(\operatorname{S}(i,2)-\operatorname{NS}(\theta,2)).\label{kjio}
\end{align}
Note that
\begin{equation}\label{jhgiuty}
0\le \operatorname{S}(i,2)-\operatorname{NS}(\theta,2)\le \operatorname{S}(k,2).\end{equation}
Furthermore, as easily seen from Lemma~\ref{klnhuighi}, (i), there exists a constant $C_2>0$ such that, for each $i\in\{3,\dots,k\}$ and $n\in\mathbb N$,
\begin{equation}\label{bhuyfu8}
\sum_{\substack{\theta\in\mathscr{NP}(k)\\ |\theta|=i}}\left| R_{\Lambda^{(n)}}\big(\theta;
M(\Lambda^{(n)};f_1),\dots, M(\Lambda^{(n)};f_k)\big)\right|\le\frac{C_2}{(V^{(n)})^3}\,. \end{equation}
Therefore, by \eqref{kjio}--\eqref{bhuyfu8},
\begin{align}
&\tau_{\Lambda^{(n)}}\big(
A(\Lambda^{(n)};f_1)\dotsm A(\Lambda^{(n)};f_1)\big)
=\sum_{\substack{\theta\in\mathscr{NP}(k)\\ |\theta|\le2}}I^{(n)}(\theta;f_1,\dots, f_k)\notag\\
&\qquad+e^{-V^{(n)}}\sum_{i=3}^k\sum_{N=2}^\infty \frac{(V^{(n)})^N }{(N-1)!}\, \big(N^{i-1}-1-(N-1)\operatorname{S}(i,2)\big)
\notag\\
&\qquad\times \sum_{\substack{\theta\in\mathscr{NP}(k)\\ |\theta|=i}} R_{\Lambda^{(n)}}\big(\theta;
M(\Lambda^{(n)};f_1),\dots, M(\Lambda^{(n)};f_k)\big)+r_2^{(n)},\label{klt87t}
\end{align}
where $\lim_{n\to\infty}r_2^{(n)}=0$.

This procedure can be iterated, which is shown in the following lemma. 
Below we will use the standard notation $(N)_j:=N(N-1)(N-2)\dotsm(N-j+1)$, the  {\it falling factorial}.

\begin{lemma}\label{klouy9i}
 For each $m\in\{2,\dots,k\}$, we have
\begin{align}
&\tau_{\Lambda^{(n)}}\big(
A(\Lambda^{(n)};f_1)\dotsm A(\Lambda^{(n)};f_k)\big)
=\sum_{\substack{\theta\in\mathscr{NP}(k)\\ |\theta|\le m}}I^{(n)}(\theta;f_1,\dots, f_k)\notag\\
&\qquad+e^{-V^{(n)}}\sum_{i=m+1}^k\sum_{N=2}^\infty \frac{(V^{(n)})^N\,K(N,i,m) }{(N-1)!}\notag\\
&\qquad\times \sum_{\substack{\theta\in\mathscr{NP}(k)\\ |\theta|=i}} R_{\Lambda^{(n)}}\big(\theta;
M(\Lambda^{(n)};f_1),\dots, M(\Lambda^{(n)};f_k)\big)+r_{m}^{(n)},\label{hfgtutr}
\end{align}
where $\lim_{n\to\infty}r_{m}^{(n)}=0$ and
\begin{equation}\label{yur87}
K(N,i,m):=N^{i-1}-1-(N-1)_1\operatorname{S}(i,2)-(N-1)_2 \operatorname{S}(i,3)-\dots-(N-1)_{m-1}\operatorname{S}(i,m).\end{equation}

In particular, for $m=k$, 
\begin{equation}\label{ghfr7}
\tau_{\Lambda^{(n)}}\big(A(\Lambda^{(n)};f_k)\dotsm A(\Lambda^{(n)};f_1)\big)
=\sum_{\theta\in\mathscr{NP}(k)}I^{(n)}(\theta;f_1,\dots, f_k)+r_k^{(n)}.\end{equation}
Here $I(\theta;f_1,\dots, f_k)$  is defined analogously to \eqref{hjfut}, \eqref{ioy89} with $\Lambda^{(n)}$ being replaced with $X$. 
\end{lemma}

\begin{proof} We prove formula \eqref{hfgtutr} by induction on $m$. We have already shown that \eqref{hfgtutr} holds for $m=2$, see \eqref{klt87t}. So assume that  \eqref{hfgtutr} holds for $m\in\{2,\dots,k-1\}$, and we have to prove \eqref{hfgtutr}
for $m+1$. The key observation here is that
\begin{equation}\label{ng8ut}
 K(N,m+1,m)=(N-1)_m,\end{equation}
 which is equivalent to the following classical identity for the Stirling numbers of the second kind:
 $$N^{m}=\sum_{j=1}^m \operatorname{S}(m,j)(N)_j,$$
 see e.g.\ \cite[Theorem 13.5]{vLW}. The rest of the proof is similar to the way we derived formula \eqref{klt87t} from \eqref{hjgcukF}.

 Indeed, for each $\theta\in\mathscr{NP}(k)$ with $|\theta|=m+1$, we have
 \begin{align}
&R_{\Lambda^{(n)}}\big(\theta;
M(\Lambda^{(n)};f_1),\dots, M(\Lambda^{(n)};f_k)\big)=\frac1{(V^{(n)})^{m+1}}\,I^{(n)}(\theta; f_1,\dots,f_k)\notag\\
&\qquad- \sum_{\substack{\pi\in\mathscr{NP}(k)\\ |\pi|\ge m+2,\, \pi\le\theta}}R_{\Lambda^{(n)}}\big(\pi;
M(\Lambda^{(n)};f_1),\dots, M(\Lambda^{(n)};f_k)\big).\label{fchydtr6}
\end{align}
Then, by \eqref{hfgtutr}--\eqref{fchydtr6},
\begin{align}
&\tau_{\Lambda^{(n)}}\big(
A(\Lambda^{(n)};f_1)\dotsm A(\Lambda^{(n)};f_1)\big)
=\sum_{\substack{\theta\in\mathscr{NP}(k)\\ |\theta|\le m+1}}I^{(n)}(\theta; f_1,\dots,f_k)\notag\\
&\qquad+e^{-V^{(n)}}\sum_{i=m+2}^k\sum_{N=2}^\infty \frac{(V^{(n)})^N\,\big(K(N,i,m)-(N-1)_m \operatorname{S}(i,m+1)\big) }{(N-1)!}\notag\\
&\qquad\times \sum_{\substack{\theta\in\mathscr{NP}(k)\\ |\theta|=i}} R_{\Lambda^{(n)}}\big(\theta;
M(\Lambda^{(n)};f_1),\dots, M(\Lambda^{(n)};f_k)\big)+r_{m+1}^{(n)}\label{huyftdr}\end{align}
where
\begin{align}
&r_{m+1}^{(n)}:=e^{-V^{(n)}}\sum_{i=m+2}^k\sum_{N=2}^\infty\, \sum_{\substack{\theta\in\mathscr{NP}(k)\\ |\theta|=i}}
\frac{(V^{(n)})^N(N-1)_m}{(N-1)!}\big(S(i,m+1)-\operatorname{NS}(\theta,m+1)\big)\notag\\
&\qquad\times R_{\Lambda^{(n)}}\big(\theta;
M(\Lambda^{(n)};f_1),\dots, M(\Lambda^{(n)};f_k)\big)
+r_{m}^{(n)}.\label{iguy8t8u}
\end{align}
Note that, in \eqref{iguy8t8u},
$$0\le S(i,m+1)-\operatorname{NS}(\theta,m+1)\le S(k,m+1).$$
Hence, analogously to \eqref{bhuyfu8}, there exists a constant $C_{m+1}>0$ such that, for each $i\in\{m+2,\dots,k\}$ and $n\in\mathbb N$,
\begin{equation}\label{tytdeyre}
\sum_{\substack{\theta\in\mathscr{NP}(k)\\ |\theta|=i}}\left| R_{\Lambda^{(n)}}\big(\theta;
M(\Lambda^{(n)};f_1),\dots, M(\Lambda^{(n)};f_k)\big)\right|\le\frac{C_{m+1}}{(V^{(n)})^{m+2}}\,. \end{equation}
By \eqref{iguy8t8u} and \eqref{tytdeyre}, $\lim_{n\to\infty}r_{m+1}^{(n)}=0$.
The definition \eqref{yur87} gives rise to
$$K(N,i,m)-(N-1)_m S(i,m+1)=K(N,i,m+1).$$
Hence, by \eqref{huyftdr},  the induction step is proven.
\end{proof}

By \eqref{ghfr7}, 
$$\lim_{n\to\infty}\tau_{\Lambda^{(n)}}\big(A(\Lambda^{(n)};f_1)\dotsm A(\Lambda^{(n)};f_1)\big)=
\sum_{\theta\in\mathscr{NP}(k)}I(\theta;f_1,\dots, f_k).$$
But by \eqref{vgjuftuf}, \eqref{gydfy6er}, and \eqref{yufr76r},
$$\tau\big(A(f_1)\dotsm A(f_k)\big)=\sum_{\theta\in\mathscr{NP}(k)}I(\theta;f_1,\dots, f_k),$$
which implies \eqref{kigt8uy6t}.
\end{proof}

\begin{proof}[Proof of Corollary \ref{kigy8u}]
The first statement trivially follows from the proof of Theorem~\ref{hgf7utr76tr6}. To prove formula \eqref{hjuft}, we proceed as follows. Using Lemma \ref{klnhuighi}, we get,  analogously to \eqref{nbfgu},
\begin{align}
&\tau_{\Lambda^{(n)}}\big(
\tilde A(\Lambda^{(n)}\times \Delta^{(n)};f_1)\dotsm A(\Lambda^{(n)}\times \Delta^{(n)};f_k)\big)\notag\\
&\quad=e^{-V^{(n)}}\sum_{i=1}^k\sum_{N=1}^\infty \frac{(V^{(n)})^N N^{i-1}}{(N-1)!}
\sum_{\substack{\theta\in\mathscr{NP}(k;2)\\ |\theta|=i}} R_{\Lambda^{(n)}}\big(\theta;
M(\Lambda^{(n)};f_1),\dots, M(\Lambda^{(n)};f_k)\big).\label{jguyt7}
\end{align}
Here $\mathscr{NP}(k;2)$ denotes the collection of all non-crossing partitions $\theta$ of $\{1,\dots,k\}$ such that each set $B\in\theta$ has at least two elements. Completely analogously to the proof of Theorem \ref{hgf7utr76tr6}, we show that the right hand side of \eqref{jguyt7} converges to
$$\sum_{\theta\in\mathscr{NP}(k;2)}R(\theta;A(f_1),\dots,A(f_k)) =\tau(\tilde A(f_1)\dotsm\tilde A(f_k)).$$
\end{proof}

\begin{center}
{\bf Acknowledgements}\end{center}
M.B. and E.L. acknowledge the financial support of the Polish
National Science Center, grant no.\ DEC-2012/05/B/ST1/00626, and of
the SFB 701 `Spectral
structures and topological methods in mathematics,' Bielefeld University. M.B. was partially supported by the MAESTRO grant DEC-2011/02/A/ST1/00119. J.L.dS., T.K., and E.L. acknowledge the
financial support from FCT --- Funda\c{c}\~{a}o para a Ci{\^e}ncia e a Tecnologia through the project UID/MAT/04674/2013 (CIMA). 
T.K. and E.L. were partially supported by the LMS Joint Research Group 'Mathematical modelling of random multicomponent systems.'


\begin{thebibliography}{99}

\bibitem{BNT} Barndorff-Nielsen, O.E., Thorbj{\o}rnsen, S.: The L\'evy--It\^o  decomposition  in free probability. Probab.~Theory Related Fields 131, 197--228 (2005) 

\bibitem{BAK} Ben Arous, G., Kargin, V.: Free point processes and free extreme values. Probab.~Theory Related Fields 147, 161--183 (2010)

\bibitem{BP}  Bercovici, H., Pata, V., with an appendix by Biane, P.:  Stable laws and domains of attraction in free
probability theory. Ann.~of Math.\ 149, 1023--1060 (1999)

\bibitem{BUS} Berezansky, Y.M., Sheftel, Z.G., Us, G.F.: Functional Analysis. Vol.~1. Birkh\"auser, Basel, Boston, London  (1996)


\bibitem{Biane} Biane, P.: Processes with free increments. Math.~Z. 227, 143--174 (1998) 


\bibitem{BL} Bo\.zejko, M., Lytvynov, E.: Meixner class of non-commutative generalized stochastic processes with freely independent values. I. A characterization. Comm.~Math.~Phys.\ 292,  99--129 (2009)


\bibitem{DOP} Di Nunno, G., {\O}ksendal, B., Proske, F.:  Malliavin calculus for L\'evy processes with applications to finance. Universitext. Springer-Verlag, Berlin  (2009)




\bibitem{DSS} Dobrushin, R.L., Sinai, Ya.G., Sukhov, Yu.M.: Dynamical systems of statistical mechanics. In: Sinai, Ya.G. (ed.) Dynamical Systems II.
Ergodic Theory with Applications to Dynamical Systems and
          Statistical Mechanics. Encyclopaedia Math.~Sci. Vol.~2, pp.~208--254. Springer, Berlin, Heidelberg (1989)



\bibitem{Ellis} Ellis, R.S.: Entropy, Large Deviations, and Statistical Mechanics.  Springer-Verlag, New York (1985)


\bibitem{GV} Gel'fand, I.M., Vilenkin, N.Ya.: Generalized Functions. Vol.~4: Applications of Harmonic Analysis.  Academic Press, New York, London  (1964)



\bibitem{HKPR} Hagedorn, D., Kondratiev, Y., Pasurek, T.,  R\"ockner, M.: Gibbs states over the cone of discrete measures. J.~Funct.~Anal.\ 264,  2550--2583 (2013)

\bibitem{Kallenberg}  Kallenberg, O.: Random Measures. Fourth edition. Akademie-Verlag, Berlin; Academic Press, London (1986)


\bibitem{KingmanCRM} Kingman, J.F.C.:
Completely random measures.
Pacific J.~Math.\ 21, 59--78 (1967)


\bibitem{Kingman}Kingman, J.F.C.: Poisson Processes. Oxford University Press, New York (1993)

\bibitem{KLV} Kondratiev, Y., Lytvynov, E., Vershik, A.: Laplace operators on the cone of Radon measures, J. Funct.\ Anal.\ 269, 2947--2976 (2015)


\bibitem{L}  Lytvynov, E.W.: Multiple Wiener integrals and non-Gaussian white noises: a Jacobi field approach. Methods Funct.~Anal.~Topology 1, no.~1, 61--85 (1995)

\bibitem{Maassen} Maassen, H.: Addition of freely independent random variables. J.~Funct.~Anal.~106, 409--438 (1992)

\bibitem{NiSp} Nica, A., Speicher, R.: Lectures on the Combinatorics of Free Probability. London Mathematical Society Lecture Note Series, vol.~335. Cambridge University Press, London (2006)

\bibitem{NZ} Nguyen, X.X., Zessin, H.:  Martin--Dynkin boundary of mixed Poisson processes. Z.~Wahrsch.~verw.~Gebiete 37,  191--200 (1976/77)

\bibitem{Ruelle} Ruelle, D.: Statistical Mechanics. Rigorous Results. W.A.~Benjamin, New York, Amsterdam (1969)



\bibitem{Speicher94} Speicher, R.: Multiplicative functions on the lattice of noncrossing partitions and free convolution. Math.~Ann.\ 298,  611-628 (1994)

\bibitem{Speicher}  Speicher, R.: Free probability theory and non-crossing partitions.
 S\'em.~Lothar.\ Combin.\  39, Art.\ B39c, 38 pp.\ (1997)  (electronic)



\bibitem{Surgailis} Surgailis, D.: On multiple Poisson stochastic integrals and associated Markov semigroups. Probab.~Math.~Statist.\ 3, 217--239 (1984)


\bibitem{TsVY} Tsilevich, N., Vershik, A.,  Yor, M.: An infinite-dimensional analogue of the Lebesgue measure and distinguished properties of the gamma process. J.\ Funct.\ Anal.\ 185, 274--296 (2001)



\bibitem{vLW} van Lint, J.H., Wilson, R.M.: A Course in Combinatorics. Cambridge University Press, Cambridge (1992)

\bibitem{V} Voiculescu, D.: Lectures on free probability. In: Bernard, P.~(ed.) Lectures on Probability Theory and Statistics. Lecture Notes in Mathematics, vol.~1738, pp.~279--349. Springer, Heidelberg  (1998)

\bibitem{Voiculescu} Voiculescu, D., Dykema, K., Nica, A.: Free Random Variables. CRM Monograph Series. No.~1.~A.M.S., Providence, RI (1992)
 \end{thebibliography}
\end{document}